\documentclass[11pt]{article}
\usepackage[square,numbers]{natbib}
\usepackage{amsthm}
\usepackage{authblk}
\usepackage{graphicx}
\usepackage{amssymb}
\usepackage{amsmath}

\newtheorem{theorem}{Theorem}[section]

\newtheorem{lemma}[theorem]{Lemma}

\begin{document}

\title{ Explicit bounds on exceptional zeroes of Dirichlet $L$-functions}
\author{Matteo Bordignon}
\affil[]{School of Physical, Environmental and Mathematical Sciences}
\affil[]{University of New South Wales Canberra }
\affil{m.bordignon@student.unsw.edu.au}
\date{\vspace{-5ex}}
\maketitle
\begin{abstract}
The aim of this paper is to improve the upper bound for the exceptional zeroes $\beta_0$ of Dirichlet $L$-functions.
We do this by improving on explicit estimates for $L'(\sigma, \chi)$ for $\sigma$ close to unity.
\end{abstract}
\section{Introduction}
The real part of the zeroes of $L(s,\chi):= \sum_{n=0}^{\infty} \chi(n)n^{-s}$ , with $\chi$ a Dirichlet character and $\Re(s)\in(0,1)$, is of high interest. It is deeply related to the size of the remainder term of the Prime Number Theorem for primes in arithmetic progression.\\
Defining $\Pi(s,q):=  \Pi_{\chi \pmod{q} } L(s,~\chi)$ an explicit result is the following, that is Theorem 1.1 of \cite{McCurley}.
\begin{theorem}[McCurley]
With $R_0=9.6459$, the function $\Pi(s, q)$ has at most one zero $\rho=\beta+it$ in the region $\beta \geq 1-(R_0\log \max\lbrace q, q|t|, 10 \rbrace)^{-1}$. This zero, if it exists, must be real and simple and must correspond to a non-principal real  character $\chi\pmod{q} $.
\end{theorem}
This zero will be called the exceptional zero $\beta_0$.\\
Platt \cite{Platt} proved that for a zero to be exceptional $q>4\cdot 10^5$ must hold, checking the Generalised Riemann Hypothesis for primitive characters to height $\max \big(\frac{10^8}{q}, \frac{A\cdot 10^7}{q}\big)$, with $A=7.5$ if the character is even, $A=3.5$ if it is odd. 
Another explicit result, by Kadiri \cite{Kadiri1}, is that for $q \leq 4\cdot10^5$ there are no zeroes in the region $\beta >~1-~(R_0\log \max\lbrace q, q|t| \rbrace)^{-1}$, with $R_1= 5.60$.

For any exceptional zero $\beta_0$ we have $\beta_0 \leq 1 -\frac{\lambda}{  q^{1/2}\log^2 q}$, where $\lambda$ can be calculated explicitly:
\begin{enumerate}
\item Liu and Wang prove $\lambda = \frac{\pi}{0.4923}$ with conductor $q'> 987$ in Theorem 3 from \cite{Liu-Wang},
\item Bennet et al. prove $\lambda = 40$ for $q \geq 3$ in Proposition 1.11 from \cite{Bennett}. 
\end{enumerate}

These results follows from the mean-value theorem, a lower bound for $L(1,\chi)$, obtained using the Class Number Formula, and an upper bound for $\left| L'(\sigma,\chi)\right|$, with $\sigma \in (\beta_0,1)$.\\
The improvement by Bennett et al. is due to better lower and upper bounds, the first obtained using more precise calculations and the second with computational aid. In Lemma 6.5 Bennett et al. obtain $\left| L'(\sigma,\chi)\right| \leqslant 0.27356 \log^2 q$ for $\chi$ primitive, $q\geq 4\cdot10^5$ and $\beta_0 \geqslant 1- \frac{1}{ 4 \sqrt{q}}$, which appears to be the best result in the literature. Thus can be improved using a better P\'{o}lya--Vinogradov inequality \citep{Frolenkov}, but this would lead to at most $\left| L'(\sigma,\chi)\right| \leqslant 0.23 \log^2 q$.
From Theorem \ref{11} we are able to obtain better upper bounds for $\left| L'(\sigma,\chi) \right|$.

\begin{theorem}
\label{40}
Assume $\chi$ is a primitive real character and $\sigma\in (\beta_0, 1)$. With $\chi$ odd, $\beta_0~\geqslant~1-~\frac{800}{ \sqrt{q}\log^2 q}$ and $q>4\cdot~10^5$, the following bound holds 
\begin{equation}
\label{123}
\left| L'(\sigma,\chi)\right| \leqslant 0.18  \log^2 q.
\end{equation} 
With $\chi$ even, $\beta_0~\geqslant ~1-~\frac{515}{ \sqrt{q}\log^2 q}$ and $q>4\cdot 10^5$, the following bound holds 
\begin{equation}
\label{6}
\left| L'(\sigma,\chi)\right|  \leqslant 0.1536  \log^2 q.
\end{equation} 
With $\chi$ even, $\beta_0~\geqslant~ 1-~\frac{80}{ \sqrt{q}\log^2 q}$ and $q>10^7$, the following bound holds 
\begin{equation}
\label{5}
\left| L'(\sigma,\chi)\right|  \leqslant 0.15 \log^2 q.
\end{equation} 
\end{theorem}
 These results need to be used together with the lower value for $q$ by Platt \cite{Platt} and the result in section A.10 from \cite{Bennett}, to get the following upper bounds for $\beta_0$.
\begin{theorem}
\label{41}
Assume $\chi$ is a non-principal real character. 
With $\chi$ odd and $ q > 4\cdot 10^5 $ , the following bound holds
\begin{equation}
\label{81}
\beta_0 \leq 1-\frac{800}{ \sqrt{q}\log^2 q}. 
\end{equation}
With $\chi$ even and $  4\cdot 10^5 < q \leq 10^7 $, the following bound holds
\begin{equation}
\label{82}
\beta_0 \leq 1-\frac{515}{ \sqrt{q}\log^2 q}. 
\end{equation}
With $\chi$ even and $ q > 10^7 $, the following bound holds
\begin{equation}
\label{83}
\beta_0 \leq 1-\frac{80}{ \sqrt{q}\log^2 q}. 
\end{equation}
\end{theorem}
One last result is obtained using Theorem \ref{11} together with the explicit version of the Burgess bound from \citep{Trevino}.
\begin{theorem}
\label{42}
With $\chi$ a primitive real character modulo $p$ prime, $\beta_0 \geqslant 1-\frac{c}{ \sqrt{p}\log^2 p}$ and $\sigma \in (\beta_0, 1)$, we have
\begin{equation*}
\left|  L'(\sigma, \chi) \right|\leq \left( \frac{1}{32} + o(1)\right) \log^2 p,
\end{equation*}
with the reminder term $o(1)$ explicit.
\end{theorem}
In Section [\ref{01}] we prove Theorems \ref{11} and \ref{13} that will allow to improve the bound on $\left|  L'(\sigma, \chi) \right|$. In Section [\ref{02}] we use  Theorem \ref{11} to prove Theorem \ref{40} and Theorem \ref{41}. In Section [\ref{03}] we apply Theorem \ref{13} to prove Theorem~\ref{42}.
\section{Preliminary results}
\label{01}
We now prove two general results that will later be applied to primitive real characters. These theorems will be used to improve the upper bound for $\left|  L'(\sigma, \chi) \right|$.
\begin{lemma}
\label{11}
Let $g(n)$ be such that for all $n$ we have $g(n)=\{-1, 0, 1\}$. 
We further assume that there is a $ M(q) \in \Re$ such~that
\begin{equation*}
 \max_k \left|\sum_{n=0}^k g(n)\right| \leq M(q).
\end{equation*}
Let $f: \Re  \rightarrow \Re $ such that $f \geq 0$, $f \rightarrow 0$, $f\in \mathcal{C'}$, $f'(x) <0$ and $\left| f'\right|  \searrow$ such~that
\begin{equation*}
\int_0^{\infty} \left| f'(x) \right|  dx < \infty .
\end{equation*}
Then we have
\begin{equation*}
\left|\sum_{n=0}^{\infty} g(n)f(n)\right|\leq  \sum_{n=0}^{\lfloor M(q)\rfloor} f(n).
\end{equation*}

\end{lemma}
\begin{proof}
Using the partial summation formula, with $f \rightarrow 0$ and the bound on $\left|\sum_{n=0}^k g(n)\right| $, we obtain
\begin{equation*}
\left|\sum_{n=0}^{\infty} g(n)f(n)\right|= \left|\int_{0}^{\infty} \left( \sum_{n \leq x}g(n) \right) (-f'(x))~ dx\right|,
\end{equation*}
with $-f'(n) > 0$.
Given that $-f'(n) > 0$ we want $\sum_{n \leq x}g(n)$ that would maximize all other possible choices and it is easy to see that this is obtained by the function $g(n)$ defined as follows 
\begin{equation*}
 g(n)=\left\{
                \begin{array}{ll}
                  1 ~\text{when}~ n \leq  \lfloor  M(q) \rfloor \\
                 0~ \text{while}~ n > \lfloor  M(q)\rfloor  .
                \end{array}
              \right.
\end{equation*}



Note that also $-g(n)$ is a maximizing function.
The result follows easily from this choice of $g(n)$.
\end{proof}

Now we give a version of Theorem \ref{11} adding an upper bound depending on $N$ that is shaped after Burgess' bound.
\begin{lemma}
\label{13}
Adding the hypothesis
\begin{equation}
\label{20}
 \left|\sum_{n=0}^{N} g(n)\right| \leq V(N)
\end{equation}
to the hypotheses in Theorem \ref{11}, and if $V(N)\leqslant \min \{N, V(N)\}$ holds true when
\begin{equation*}
C_1(q)\leqslant N\leqslant C_2(q),
\end{equation*}
with $C_1(q), C_2(q) \in \mathbb{N}$, then $\left|\sum_{n=0}^{\infty} g(n)f(n)\right|$ has as an upper bound 
\begin{equation}
\label{100}
\sum_{n=0}^{ C_1(q)} f(n)+M(q) f(C_2(q))-C_1(q) f(C_1(q))-\int_{C_1(q)}^{C_2(q)} V(x) f'(x) dx .
\end{equation}

\end{lemma}
\begin{proof}
The proof is similar to that of Theorem \ref{11} but with the maximizing function $g(n)$ that assumes the value zero enough times to make condition (\ref{20}) hold.\\
Now the result follows applying partial summation to $\left| \sum_{n=0}^{ \left\lfloor  M(q) \right\rfloor } g(n) f(n) \right|$.
\end{proof}
Note that the previous results can surely be further modified adding different conditions.
\section{Upper bound for the exceptional zero $\beta_0$ }
\label{02}
A standard way to get an upper bound on the exceptional zero is to use the mean value theorem to obtain
\begin{equation*}
L(1,\chi)= L(1,\chi)-L(\beta_0,\chi)=L'(\sigma,\chi)(1-\beta_0) ~\text{for~some }~ \sigma \in (\beta_0,1),
\end{equation*}
where
\begin{equation*}
L'(\sigma,\chi)= \frac{d}{d\sigma}L(\sigma,\chi)=-\sum_{n=1}^{\infty} \chi(n) \frac{\log n}{n^{\sigma}} ~\text{with }~ \sigma > 0.
\end{equation*}
Now we see that
\begin{equation}
\label{2}
1-\beta_0=\frac{L(1,\chi)}{| L'(\sigma,\chi)|}
\end{equation}
and thus we are left to obtain a lower bound for $L(1,\chi)$ and an upper bound for $| L'(\sigma,\chi) |$ when $\sigma \in (\beta_0,1)$.
\subsection{Lower bound for $L(1,\chi)$}
We start fixing $q > 4\cdot 10^5$.
We use that every real primitive character can be expressed using the Kroneker symbol, as $\chi(n)= (\frac{d}{n})$, with $q= \left| d\right| $.
First we consider $d < 0$, and thus $ d < -4\cdot10^5$. Dirichlet's class number formula then gives
\begin{equation*}
L(1,\chi)=\frac{2\pi h(\sqrt{d})}{ \omega_d \sqrt{| d |}},\quad \text{with}~ \chi(-1)=-1,
\end{equation*}
with $h(\sqrt{d})$ the class number of $ \mathbb{Q}(\sqrt{d})$ and $\omega_d$ is the number of roots of unity in $ \mathbb{Q}(\sqrt{d})$, which is known to be 2 when $d< -3$. From Table 4 in Watkins \cite{Watkins}, and $ | d | > 319867$ we obtain that $h(\sqrt{-q} )\geq 46$, that gives
\begin{equation}
\label{4}
L(1,\chi)=\frac{2\pi h(\sqrt{d})}{ \omega_d \sqrt{| d |}} \geq \frac{46 \pi}{\sqrt{q}},\quad\text{with}~ \chi(-1)=-1.
\end{equation}
Now we consider $d>0$. Dirichlet's class number formula gives
\begin{equation}
\label{77}
L(1,\chi)=\frac{ h(\sqrt{d} )\log \eta_d}{  \sqrt{ d }},\quad \text{with}~ \chi(-1)=1,
\end{equation}
where $\eta_d=(v_0+u_0\sqrt{d})/2$ , with $v_0$ and $u_0$ the minimal positive integers satisfying $v_0^2-du_0^2=4$.\\
From A.10 in \cite{Bennett} we have that 
\begin{equation}
\label{7}
h(\sqrt{d} )\log \eta_d > 79.2177
\end{equation} when $4\cdot10^5\leqslant d \leqslant 10^7$. Bennett et al. then prove that $h(\sqrt{d} )\log \eta_d >12$ for $ q > 10^7$, thus obtaining from (\ref{77}) that for all $d > 4\cdot10^5$
\begin{equation}
\label{8}
L(1,\chi)\geq \frac{12}{ \sqrt{q}},\quad \text{with}~ \chi(-1)=1.
\end{equation}

\subsection{Upper bound for $\left| L'(\sigma,\chi) \right|$ and proof of Theorem \ref{41}}
Here we assume $\sigma \in (\beta_0, 1)$ and $\beta_0 \geqslant 1-\frac{c}{ \sqrt{q}\log^2 q}$, with $c$ a positive constant to be chosen later.
Then we define
\begin{equation*}
S_0(\chi):=\max_{N} \left| \sum_{n=1}^N \chi (n) \right|, \quad S(\chi):=\max_{N, M} \left| \sum_{n=M+1}^{M+N} \chi (n) \right|.
\end{equation*}

Note that, taking $M=0$, we have $S_0(\chi) \leqslant S(\chi)$. Then we always have, from the triangle inequality, that $S(\chi) \leqslant 2S_0(\chi)$ and, when the character is even, $S(\chi) = 2S_0(\chi)$ see \cite[p.~533]{Pomerance}.
Now we apply Theorem \ref{11} with a primitive real character as $g(n)$ and $f(n)=\frac{\log n}{n^\sigma}$, but with the sum starting from $n=2$, and we get
\begin{equation}\
\label{1}
\left| L'(\sigma,\chi)\right| \leqslant 
 \sum_{n=2}^{S_0(\chi)} \frac{\log n}{n^\sigma}.
\end{equation}
 With $d < S_0(\chi)$, we further obtain by partial summation
 \begin{equation}
\label{12}
\left| L'(\sigma,\chi)\right| \leqslant S_0(\chi)^{1-\sigma}
\left( \frac{1}{2}\log^2 S_0(\chi) - \frac{1}{2}\log^2 d +\sum_{n=2}^{d} \frac{\log n}{n}\right),
\end{equation}
and fixing $d= 100$ 
 \begin{equation*}
 - \frac{1}{2}\log^2 d +\sum_{n=2}^{d} \frac{\log n}{n}< 0.
\end{equation*}
This number we thus obtain is so small that we omit it in what follows.\\
We focus first on odd characters. 
Here we use the following bound, from \cite[p.~278]{Frolenkov}, for a primitive character
\begin{equation*}
S(\chi)\leq \frac{\sqrt{q}}{\pi}\left( \log(\pi \sqrt{q}+10.15)+1.4326 \right).
\end{equation*}
Plugging this bound in (\ref{12}) and using it as an upper bound for $S_0(\chi)$, we obtain
\begin{equation*}\
\left| L'(\sigma,\chi)\right| \leqslant S_0(\chi)^{1-\sigma} \left( \frac{1}{2}\log^2\left( \frac{\sqrt{q}}{\pi}\left( \log(\pi \sqrt{q}+10.15)+1.4326 \right)\right)\right)\cdot
\end{equation*}

Remembering $\beta_0 \geqslant 1-\frac{c}{ \sqrt{q}\log^2 q}$, choosing $c=800$ and with $q >~4\cdot 10^5$, we obtain
\begin{equation}
\label{3}
\left| L'(\sigma,\chi)\right| \leqslant 0.18 \log^2 q.
\end{equation} 
Now we focus on even characters.
The bound from \cite[p.~278]{Frolenkov} is
\begin{equation*}
S(\chi)\leq \frac{4}{\pi^2}\sqrt{q}\left( \log(\frac{\pi^2}{4} \sqrt{q}+10.15)+1.4326 \right).
\end{equation*}
Remembering that for even characters we have $S(\chi) = 2S_0(\chi)$, we plug $\frac{S(\chi)}{2}$ in (\ref{12}) to obtain
\begin{equation*}
\label{101}
\left| L'(\sigma,\chi)\right| \leqslant S_0(\chi)^{1-\sigma}\left( \frac{1}{2}\log^2\left( \frac{2}{\pi^2}\sqrt{q}\left( \log(\frac{\pi^2}{4} \sqrt{q}+10.15)+1.4326 \right)\right)\right)\cdot
\end{equation*}

Remembering $\beta_0 \geqslant 1-\frac{c}{ \sqrt{q}\log^2 q}$  we obtain, with $q~>~4\cdot~10^5$ and $c=515$, 
\begin{equation*}
\left| L'(\sigma,\chi)\right| \leqslant 0.1536  \log^2 q
\end{equation*} 
and, with $q >  10^7$ and $c=80$,
\begin{equation*}
\left| L'(\sigma,\chi)\right| \leqslant 0.15 \log^2 q.
\end{equation*} 
It is interesting to note, aiming to improve the above result, that we have  
\begin{equation}
\label{8}
\left| L'(\sigma,\chi)\right| \leqslant \left (\frac{1}{8}+o(1)\right) \log^2 q.
\end{equation} 
Now Theorem \ref{41} follows easily. We just need to prove the theorem for primitive real characters, indeed if $\chi \pmod{q}$ is induced by some primitive real character $\chi' \pmod{q'}$, then the primitive case yields
\begin{equation*}
\beta_0 \leq 1-\frac{\lambda}{\sqrt{q'}\log^2 q'}\leq 1-\frac{\lambda}{\sqrt{q}\log^2 q}.
\end{equation*}
Thus (\ref{81}) follows from (\ref{2}), (\ref{123}) and (\ref{4}); (\ref{82}) from (\ref{2}), (\ref{6}), (\ref{77}) and (\ref{7}); and (\ref{83}) from (\ref{2}), (\ref{5}) and (\ref{8}).
\section{Another upper bound for the exceptional zero }
\label{03}
We now prove Theorem \ref{42} applying Theorem \ref{13} to a real primitive character modulo $p$ prime.\\
The bound (\ref{20}) will be the explicit Burgess bound from Theorem 1.4. in \citep{Trevino} that states that, for any $p\geqslant 10^7$ and $c_1(r)$ from Table 1 \citep[p.~1655]{Trevino} , we have
\begin{equation*}
\left|\sum_{n=M}^{M+N} \chi(n)\right| \leq c_1(r) N^{1-\frac{1}{r}} p^{\frac{r+1}{4r^2}} \log^{\frac{1}{r}}p.
\end{equation*}
We have $c_1(r) N^{1-\frac{1}{r}} p^{\frac{r+1}{4r^2}} \log^{\frac{1}{r}}p\leq \min \{ N, S_0\left( \chi \right) \}$ when
\begin{equation*}
c_1(r)^{r} p^{\frac{(r+1)}{4r}} \log p \leq N \leq \left(\frac{c_0}{c_1(r)}\right)^{1+\frac{1}{r}} p^{\frac{2r^2-r-1}{2r(r-1)}} \log p,
\end{equation*}
where $c_0$ is such that $ S_0\left( \chi \right) \leq c_0 p^{\frac{1}{2}} \log p$.
Note that it is possible to improve the result using the explicit P\'{o}lya--Vinogradov inequality from \cite{Frolenkov}, but the result would not improve considerably.\\
Using these explicit bounds from Theorem \ref{13}, and $\beta_0 \geqslant~1-~\frac{c}{ \sqrt{p}\log^2 p}$, we obtain
\begin{equation*}
\left| L'(1, \chi) \right| \leq \left( \frac{(r+1)^2}{32r^2} + o(1)\right) \log^2 p,
\end{equation*}
with $o(1)$ explicit and easy to compute. This bound is asymptotically stronger than (\ref{8}), but results worst for small $p$ due to the size of the constant $c_1(r)$ in the explicit version of Burgess bound.\\
Theorem \ref{42} follows taking $r$ sufficiently large and noting that $c_1(r)$ is decreasing.

\section*{Acknowledgements} 
I would like to thank my supervisor Tim Trudgian for his kind help and his sharp suggestions in developing this paper and an anonymous reviewer for the useful comments.


\begin{thebibliography}{9}

\bibitem{Bennett}
M.~A Bennett, G. Martin, K. O'Bryant,  and A. Rechnitzer. 
\textit{Explicit bounds for primes in arithmetic progressions.}
ArXiv e-prints, Nov. 2018.
arxiv:1802.00085.
To appear in Illinois J. Math.

\bibitem {Frolenkov}
D. A. Frolenkov and K. Soundararajan.
\textit{A generalization of the {P}\'olya-{V}inogradov inequality.}
Ramanujan J., 31(3):271--279, 2013.
  
\bibitem{McCurley}
K.J. McCurley
\textit{Explicit zero-free regions for the Dirichlet L-functions}
J. Number Theory, 19 (1) (1984), pp. 7-32
   
\bibitem{Kadiri1}
H. Kadiri.
\textit{Explicit zero-free regions for the {D}irichlet {$L$}-functions.}
Mathematika, 64(2):445--474, 2018.
  
\bibitem{Liu-Wang}
M.-C. Liu and T. Wang.
\textit{Distribution of zeros of {D}irichlet {$L$}-functions and an
              explicit formula for {$\psi(t,\chi)$}.}
Acta Arith., 102(3):261--293, 2002.
  
\bibitem{Platt}
D. J. Platt.
\textit{Numerical computations concerning the {GRH}.}
Math. Comp., 85(302):3009--3027, 2016.
  
\bibitem{Pomerance}
C. Pomerance.
\textit{Remarks on the {P}\'olya-{V}inogradov inequality.}
Integers, 11(4):531--542, 2011.
 
\bibitem{Trevino}
E. Trevi\~no.
\textit{The {B}urgess inequality and the least {$k$}th power
              non-residue.}
Int. J. Number Theory, 11(5):1653--1678, 2015.
 
\bibitem{Watkins}
M. Watkins. 
\textit{Class numbers of imaginary quadratic fields.}
Math. Comp., 73(246):907--938, 2004.
 
\end{thebibliography}
\end{document}